\documentclass[11pt,reqno]{amsart}

\usepackage{amsfonts, amssymb, amsthm, amsmath, bbm, wasysym, enumerate, amsxtra, latexsym, fullpage, amscd, graphics, epic, sansmath, mathtools, multicol, url, hyperref, bbm, blindtext, wasysym, comment, datetime, stmaryrd,mathrsfs}
\usepackage[shortlabels]{enumitem}
\usepackage[all,cmtip]{xy}
\usepackage[dvipsnames]{xcolor}
\usepackage[makeroom]{cancel}  
\usepackage[enableskew]{youngtab}
\usepackage{tikz}
\usetikzlibrary{cd}

\usepgflibrary{shapes.geometric}

\tikzstyle{V}=[draw, fill =black, circle, inner sep=0pt, minimum size=1.5pt]
\tikzstyle{C}=[draw, fill =white, circle, inner sep=0pt, minimum size=1.5pt]
\tikzstyle{over}=[draw=white,double=black,line width=2pt, double distance=.5pt]
\usetikzlibrary{arrows, positioning, calc, chains, cd}
\tikzset{
	ch/.style={circle,draw,on chain,inner sep=2pt},
	chj/.style={ch,join},
	every path/.style={shorten >=4pt,shorten <=4pt}
	}

\numberwithin{equation}{section}
\usepackage[margin = 1in]{geometry}
\theoremstyle{definition}

\newtheorem{theorem}{Theorem}[section]

\newtheorem{proposition}[theorem]{Proposition}

\newtheorem{definition}[theorem]{Definition}
\newtheorem{remark}[theorem]{Remark}

\newtheorem{example}[theorem]{Example}

\def\<{\langle}
\def\>{\rangle}

\newcommand{\card}{\mathrm{card}}

\newcommand{\diam}{\text{diam}}

\newcommand{\st}{\textup{st}}

\allowdisplaybreaks

\title{Nonstandard approach to Hausdorff outer measure}
\author[Mee Seong Im]{Mee Seong Im}
\address{Department of Mathematical Sciences, United States Military Academy, West Point, NY 10996}
    \email{meeseongim@gmail.com}
 
\begin{document}

\begin{abstract} 
We use nonstandard techniques, in the sense of Abraham Robinson, to give the exact Hausdorff outer measure.  

\end{abstract}  
 
\maketitle 

\section{Introduction}

Developing a notion of dimension and measuring sets is foundational in mathematics (cf.~\cite{MR2257838,Fremlin-V1,Fremlin-V2,Fremlin-V3,Fremlin-V4,Fremlin-V5}). 
Originally studied by B. Riemann, they have been further developed by H. Schwarz, F. Klein, D. Hilbert, H. Lebesgue, F. Hausdorff, and others. 

In this manuscript, we focus on Hausdorff measure, which is well-defined for \textit{any} set. That is, a $d$-dimensional Hausdorff measure of a measurable subset of $\mathbb{R}^n$ is proportional to the dimension of the set. In particular, $d$-dimensional Hausdorff measure exists for any real number $d\geq 0$ (so $d$ is not necessarily an integer). 
This implies that the Hausdorff dimension of a set is greater than or equal to its topological dimension, and less than or equal to the dimension of the metric space imbedding the set (thus it is a refinement of an integral dimension).

We use a nonstandard approach to study Hausdorff measure since nonstandard techniques provide a richer insight to standard objects and sets. 
In particular, we show how discrete measure (cf. Definition~\ref{defn:h-delta-s}) gives rise to Hausdorff outer measure, which is our main result:  
\begin{theorem}\label{thm:main} 
Let $A$ be a subset in $[0,1]$. Let $\mathcal{H}^s(A)$ be the Hausdorff outer measure of dimension $s$. Then 
\begin{enumerate}[label=(\alph*)]
\item\label{item:a} $\mathcal{H}^s(A) \leq \inf\{\st \: h_{\delta}^s(B):  \text{internal} \: B \supseteq \st^{-1}(A)\}$ for all infinitesimal $\delta$, and 
\item\label{item:b} $\mathcal{H}^s(A) = \lim_{\delta\rightarrow 0} \inf \{ \st \: h_{\delta}^s(B): 
\text{internal}\: B \supseteq \st^{-1}(A)\}$, where $\delta$ ranges over standard positive values in $\mathbb{R}$ in the limit.  
\end{enumerate} 
\end{theorem}

Theorem~\ref{thm:main} introduces \textit{counting} methods that calculate the correct Hausdorff measure. Other methods, such as Minkowski, \textit{i.e.}, box-counting, methods, have been proven to overestimate or underestimate the correct dimension of a set, particularly if it is irregular (cf. \cite{gorski2012accuracy,mainieri1993equality}). In fact, if a set or its complement is not self-similar, then the box-counting method fails since there is no dimension for which the limit converges. 
 
Secondly, the nonstandard version is easier to compute than the standard Hausdorff version. This is because the discrete measure of $\mathcal{H}^s(A)$ has a fixed $\delta$, not a varying one. Also, taking the supremum over any $\delta$ is omitted to compute $\mathcal{H}^s(A)$ (see Definition~\ref{defn:s-dim-Hausdorff-outer-measure}) since the process of taking the supremum over all $\delta$'s has already been applied when choosing our $\delta$. So although this operation is omitted, we still obtain the accurate Hausdorff dimension, which is defined for any set. 

Throughout this manuscript, we assume that the set $\mathbb{N}$ of natural numbers includes $0$. 

\subsection{Summary of the sections}
\label{subsection:summary}
In \S\ref{section:nonstandard-analysis}, we introduce basic and key ideas in nonstandard analysis needed to prove Theorem~\ref{thm:main}.  In \S\ref{section:Hausdorff-measure}, we give a background on Hausdorff measure. In \S\ref{section:nonstandard-techniques}, we develop a notion of computing measure in the nonstandard universe. 
We give a relation between $\mathcal{H}^s(A)$ and Lebesgue outer measure $\overline{\lambda}(A)$ for nice sets $A$ (cf.~\eqref{eqn:Hausdorff-outer-measure-Lebesgue-outer-measure}), which provides an alternative (discrete) way to obtain Hausdorff outer measure.  
In \S\ref{section:proof-main-theorem}, we prove Theorem~\ref{thm:main}. 
Finally, we conclude with an example in \S\ref{section:thm-inequality-example}, showing that Theorem~\ref{thm:main}\ref{item:a} cannot be replaced by an equality.

\subsection*{Acknowledgment}
The author would like to thank Richard Kaye, for his patience and numerous helpful discussions. 
M.S.I. was partially supported by the School of Mathematics at the University of Birmingham in the United Kingdom. 
 
\section{Nonstandard analysis}
\label{section:nonstandard-analysis}

We give a background on the theory of nonstandard analysis (also see,~\textit{e.g.},~\cite{Robinson96,Robert03,AHFL86}).

\subsection{Ultrafilters}
\label{subsection:ultrafilters}

A filter over a set $I$ is a nonempty collection $\mathscr{U}$ of subsets of $I$ that satisfies the following: 
\begin{enumerate}
\item if $A,B\in \mathscr{U}$, then $A\cap B\in \mathscr{U}$, and 
\item if $A\in \mathscr{U}$ and $B\supseteq A$, then $B\in \mathscr{U}$. 
\end{enumerate}  
Filters can be used to define a notion of \textit{large} subsets of $I$. 
That is, a set $A\subseteq \mathbb{N}$ is large if and only if 
\begin{enumerate}
\item a finite subset of $\mathbb{N}$, including the empty set, is not large, 
\item the set of natural numbers is large, 
\item two subsets of $\mathbb{N}$ are large, then all supersets of their intersection are also large, 
\item its complement $A^c$ is not large. 
\end{enumerate} 
This notion of largeness is a special kind of filter called an ultrafilter.

\begin{definition}
A nonprincipal ultrafilter $\mathscr{U}$ of subsets of a nonempty set $I$ contains all large subsets of $I$.  
\end{definition}
It is a maximal proper filter since it cannot contain any more subsets without including the empty set and becoming improper. 
Each infinite set $I$ contains at least one maximal nonprincipal proper filter, and 
if there are more than one nonprincipal ultrafilters, then they are isomorphic to each other. 
Let $a=(a_0,a_1,\ldots), b=(b_0,b_1,\ldots)\in \mathbb{R}^{\mathbb{N}}$, sequences of real numbers. We say $a=b$ if and only if $a_i=b_i$ for a large set of $i\in \mathbb{N}$, 
and we write 
$a\sim b$ if $\{ i: a_i=b_i\}\in \mathscr{U}$.

Let $A\subseteq \mathbb{R}$ be a set. The starred version ${}^*A$ of $A$ is defined as $A^{\mathbb{N}}/\!\!\sim$, where $\sim$ is the equivalence relation given above; we say ${}^*A$ is the nonstandard version of $A$. 
With a slight abuse of notation, we write 
$a = [a_0,a_1,\ldots]\in {}^*\mathbb{R}$ to denote the equivalence class of $a=(a_0,a_1,\ldots)$.

\subsection{Nonstandard real and natural numbers}

Let ${}^*\mathbb{R}$ be the set of nonstandard reals. The algebraic operations on ${}^*\mathbb{R}$ are componentwise, and the partial order $<$ on ${}^*\mathbb{R}$ is defined to be $a<b$ if and only if $\{i: a_i < b_i \}\in \mathscr{U}$.  
Note that $\mathbb{R}\hookrightarrow {}^*\mathbb{R}$ via the constant-valued sequence embedding $r\mapsto {}^*r= [r,r,r,\ldots]$.

We will now describe hyperreal numbers.    
Suppose  
$a = [a_0,a_1,a_2,\ldots]\in {}^*\mathbb{R}$. Then for a large $i\in \mathbb{N}$,  
$a$ is a  
\begin{enumerate}  
\item positive infinitesimal if $0\leq a_i <t$ for every positive real $t$, 
\item finite if $s < a_i<t$ for some $s,t\in \mathbb{R}$, 
\item positive infinite if $s< a_i$ for every $s\in \mathbb{R}$, and 
\item standard if $a_i=a_j$ for any $i$ and $j$. 
\end{enumerate}


\begin{example} 
Let $a=[1/n]_{n\in \mathbb{N}}$ and $b = [1/n^2]_{n\in \mathbb{N}}$ be hyperreals for $n\not=0$. Then $a$ is an infinitesimal since $\{n: |1/n|<x \}\in \mathscr{U}$ for any positive real $x$. 
By the same reason, $b$ is also an infitesimal. 
The only difference between $a$ and $b$ is their rate of convergence: $a \geq b >0$. 
\end{example}

\begin{remark}  
The number ${}^*0$ is the only number that is both an infinitesimal and standard.  
\end{remark} 

We will also say a hyperreal number $x\in {}^*\mathbb{R}$ is an infinite real number if $x>{}^*n$ for every ${}^*n=[n,n,n,\ldots]\in {}^*\mathbb{N}$. 
We will write $x\in {}^*\mathbb{R}\setminus \mathbb{R}$ or $x> \mathbb{R}$ to say that $x$ is an infinite real number. Similarly, we will write $N\in {}^*\mathbb{N}\setminus \mathbb{N}$ or $N>\mathbb{N}$ to say that $N$ is an infinite natural number.

\subsection{Nonstandard superstructure}
Let $A$ be a set. Let $\{ V_i(A)\}_{i\in \mathbb{N}}$ be the sequence defined by $V_0(A)=A$ and 
$V_{i+1}(A)=V_i(A) \cup \mathscr{P}(V_i(A))$, where $\mathscr{P}(V_i(A))$ is the power set of $V_i(A)$. Then $V(A) := \bigcup_{i\in \mathbb{N}}V_i(A)$, the superstructure over $A$. The rank of $x\in V(A)$ is the smallest $k$ such that $x\in V_k(A)$. 

Let $V$ be the collection of all (pure) sets, \textit{i.e.}, elements of pure sets are hereditary, and let $\in$ be the universal symbol \textit{the element of}. Denote $\mathscr{V}=\langle V,\in \rangle$, the first-order structure in set theory. 
 We define the ultrapower of $\mathscr{V}$ as 
${}^*\mathscr{V} := \prod_{\mathscr{U}} \mathscr{V} \succ \mathscr{V}$, where $\mathscr{U}$ is an ultrafilter.

The $*$-transformation is a technique to convert a language from the standard universe to the nonstandard world. That is, if $*$ proceeds an element, set, or function, then it represents that the element or set lives in the nonstandard world, or that the operation is computed in the nonstandard universe.   
\begin{theorem}[\L o\'s' Theorem] 
Let $\mathscr{V}=\langle V,\in \rangle$, the superstructure of $V$. Let $\theta(a_0,a_1,\ldots, a_{n-1})$ 
be a first-order statement, and let $\theta({}^*a_0,{}^*a_1,\ldots, {}^*a_{n-1})$ be its $*$-transformation. 
Then 
$\theta(a_0,a_1,\ldots,a_{n-1})$ is true in $\mathscr{V}$ if and only if $\theta({}^*a_0,{}^*a_1,\ldots, {}^*a_{n-1})$ is true in ${}^*\mathscr{V}$. 
\end{theorem}

\subsection{Internal sets}
 
A subset $A$ of the nonstandard universe ${}^*V$ is called internal if there is a set $B\subseteq V$ such that 
\[ 
A = \{ [a_0,a_1,\ldots]\in {}^*V: \{i : a_i\in B \} \in \mathscr{U} \}. 
\]  
In particular, $A\subseteq {}^*\mathbb{R}$ is internal if and only if it is ${}^*B$ for some $B\subseteq \mathbb{R}$.

\begin{example}
A finite subset ${}^*A\subseteq {}^*\mathbb{R}$ such as $\{ {}^*0, {}^*1, {}^*2 \} = {}^*\{ 0,1,2\}$ 
is internal. 
\end{example}

To build internal sets, produce a sequence $\{ A_i\}_{i\in \mathbb{N}}$ of standard sets such that $A=[\{ A_i\}_{i\in \mathbb{N}}]\subseteq {}^*\mathscr{V}$. To generate an internal function $f$ that maps an internal set $A$ to an internal $B$, form a sequence of functions $f=\{ f_i\}_{i\in \mathbb{N}}$ such that each $f_i$ is well-defined between $A_i$ and $B_i$ for each $i$. 
Hence we have $f:A\rightarrow B$ if and only if $f_i:A_i\rightarrow B_i$ for almost all $i$, \textit{i.e.}, for a large subset of the natural numbers. 

 Sets which are not internal are called external. 
Since basic principles of mathematics are broken down for external sets when moving between standard and nonstandard worlds, they are rarely of our interest. 

To see how the reals are embedded in the hyperreals, we introduce monads. 

\begin{definition}
\label{defn:monads}
The standard part $\st(x)$ of a finite $x\in {}^*\mathbb{R}$  
is the unique $a\in \mathbb{R}$ that is closest to $x$, \textit{i.e.}, 
\[ 
\st(x) := \inf\{ a\in \mathbb{R}: {}^*a \geq x \} = \sup\{ a\in \mathbb{R}: {}^*a\leq x\}. 
\] 
The set consisting of $x\in {}^*\mathbb{R}$ such that $\st(x)=a$ is called the monad of $a$, and is written as 
\[ 
\st^{-1}(a)= \mu(a) = \{ x\in {}^*\mathbb{R}: \st(x)=a \}. 
\] 
\end{definition}
 
Since the monad of $0$ is not first-order definable, $\mu(0)$ is not internal. 
More generally, we have the following: 
\begin{proposition}\label{prop:monad-not-internal} 
Let $a\in \mathbb{R}$ and let $\mu(a)$ be the monad of $a$. Then $\mu(a)$ is not internal. 
\end{proposition} 

%
%
%
%
%
%


\subsection{Overspill principle and saturation}
\label{subsection:overspill-principle}
In this section, we discuss two theorems frequently used in nonstandard analysis. 
\begin{theorem}[Overspill principle] 
Let $A\subseteq {}^*\mathbb{R}$ be a nonempty internal subset containing arbitrary large finite elements. Then $A$ contains an infinite element. 
\end{theorem}

\begin{proof}
Let $A$ be a nonempty internal subset of ${}^*\mathbb{R}$. If $A$ is unbounded, then $A$ must contain at least one infinite element, and we are done. 
So suppose $A$ is bounded. Let $a$ be the least upper bound of $A$. Since $A$ contains arbitrary large finite elements, $a$ must be infinite. If there is no $x\in A$ such that $a-\varepsilon \leq x\leq a$ for some positive $\varepsilon$, then $a-\varepsilon$ is the least upper bound of $A$, which is a contradiction. 
Thus, there is some $x\in A$ such that $a-\varepsilon \leq x\leq a$, and we see that $A$ contains an infinite element. 
\end{proof} 

An alternative way to think of the overspill principle is that if a statement is true for all infinitesimals, then it is true for some standard positive number.  

\begin{proposition}
Let $\mathbb{R}$ be a domain in $\mathscr{V} = \langle V, \in \rangle$. 
Suppose $A\subseteq {}^*\mathbb{R}$ is an internal subset and each $x\in A$ is a finite element. Then $A$ contains a least upper bound. 
\end{proposition}

The overspill principle is also used to distinguish sets that are not internal to ${}^*\mathscr{V}$ for if they were internal, then we would be assuming $\mathbb{R}={}^*\mathbb{R}$, which is clearly false. 

\begin{theorem}[Saturation]
\label{theorem:saturation}
Let $\{A^i \}_{i\in \mathbb{N}}$ be a sequence of internal sets satisfying $\bigcap_{i=0}^{n} A^i\not=\varnothing$ 
for each $n\in \mathbb{N}$, where each $A^i=[A_0^i,A_1^i,A_2^i,\ldots]$. 
Then $\bigcap_{i\in \mathbb{N}}A^i\not=\varnothing. $
\end{theorem}
Theorem~\ref{theorem:saturation} is also known as $\aleph_1$-saturation (cf. \cite{Cutland97}).

\section{Hausdorff measure}
\label{section:Hausdorff-measure} 
We begin with some preliminary definitions. Also see \cite{Federer69,Hausdorff18,Ott02,Rogers70} for further background on Hausdorff measure.  
Let $d:\mathbb{R}^n\times \mathbb{R}^n\rightarrow \mathbb{R}$ be the Euclidean magnitude 
$d(x,y)=|\!|x-y|\!|:= \left(\sum_{i=1}^{n}(x_i-y_i)^2\right)^{1/2}$ for $x,y\in \mathbb{R}^n$. 
Given $U\subseteq \mathbb{R}^n$, 
the diameter of $U$ is defined as 
$\diam(U)=\sup\{ d(x,y): x,y\in U\}$. 
Now let $A\subseteq \mathbb{R}^n$. For any $U_i\subseteq \mathbb{R}^n$, the collection $\{ U_i: i\in \mathbb{N}\}$ 
is a $\delta$-cover of $A$ if 
\begin{enumerate}
\item $\bigcup_{i\in\mathbb{N}} U_i \supseteq A$, and  
\item $0\leq \diam(U_i)\leq \delta$ for each $i\in \mathbb{N}$. 
\end{enumerate}
A $\delta$-covering of a set is chosen such that the differences 
$\bigcup_i U_i\setminus A$ and 
$A\setminus \bigcup_i U_i$ are negligible sets\footnote{ Negligible sets form an ideal, and in this paper, we assume that negligible sets are $\sigma$-ideal, \textit{i.e.}, a countable union of negligible sets is negligible.}.

\begin{definition}
Let $A$ be a subset of $\mathbb{R}^n$, and let $s,\delta >0$. 
The Hausdorff $\delta$-measure of $A$ is defined as  
\[ 
\mathcal{H}_{\delta}^s(A) = \inf 
\left\{ 
\sum_{i=0}^{\infty} \diam(U_i)^s
\right\}, 
\]   
where the infimum is taken over every $\delta$-cover $\{ U_i\}_{i\in \mathbb{N}}$ of $A$. 
\end{definition}
We will simply write Hausdorff $\delta$-measure as Hausdorff measure. 

\begin{example}
\label{example:Cantor-classical-Hausdorff}
Let $C$ be the Cantor set on the unit interval. That is, it is generated by cutting the middle $1/3$ from each of the previous connected line segments. So letting $C_0=[0,1]$, $C_1=[0,1/3]\cup [2/3,1]$, 
$C_2 = [0,1/9]\cup [2/9,3/9]\cup [6/9,7/9]\cup [8/9,1]$, $\ldots$, 
$C_m=\bigcap_{i=0}^m C_i$. At $m$-th iteration, there are $2^m$ intervals of length $3^{-m}$ to cover $C_m$. So 
\[ 
\mathcal{H}_{\delta}^s(C) = 2^m 3^{-ms}, 
\]  which implies the critical value $s$ must be $\log 2/\log 3$.  
\end{example} 

\begin{definition}
\label{defn:s-dim-Hausdorff-outer-measure}
Let $A\subseteq \mathbb{R}^n$, and let $s >0$. 
Then $s$-dimensional Hausdorff outer measure of $A$ is
\[ 
\mathcal{H}^s(A) = \lim_{\delta\rightarrow 0} \mathcal{H}_{\delta}^s(A) = \sup_{\delta >0} \mathcal{H}_{\delta}^s(A). 
\]  
\end{definition}

The Hausdorff measure $\mathcal{H}_{\delta}^s(A)$ increases as $\delta$ decreases. 
In fact, $\mathcal{H}^s(A)$ is a nonincreasing function as $s\rightarrow \infty$.
 
\begin{definition}
\label{defn:Hausdorff-dim}
Let $A$ be a subset of $\mathbb{R}^n$. The Hausdorff dimension of $A$ is 
\[  
\dim_{\mathcal{H}}(A) = \inf \{ s: \mathcal{H}^s(A)=0 \}
= \sup \{ s: \mathcal{H}^s(A)=\infty\}. 
\] 
\end{definition} 
An interpretation of the infimum in Definition~\ref{defn:Hausdorff-dim} is as follows: 
if $s$ is greater than the actual Hausdorff dimension of $A$, then $\mathcal{H}^s(A)=0$, which is the reason for us to 
take the infimum over all such $s$. 
A similar argument holds for the second equality.

\section{Nonstandard analysis techniques on Lebesgue and Hausdorff measure}
\label{section:nonstandard-techniques}
In this section, we develop a notion of measure in the nonstandard world by using nonstandard analysis techniques to explore various ways to measure a set. 

Consider the unit interval, and let $\Omega=\{0,1/N,\ldots, i/N,\ldots, 1 \}$, where $N\in {}^*\mathbb{N}\setminus \mathbb{N}$ is a nonstandard natural. 
Note that the standard map for $\Omega$ sends all points to the interval $[0,1]$.  
Recall from Proposition~\ref{prop:monad-not-internal} that, for some set $A\subseteq [0,1]$,   
$\st^{-1}(A)\subseteq \Omega$ is not necessary internal. In particular, studying the cardinality $\card(\st^{-1}(A))$ is meaningless. Hence we will approximate this set using internal sets. 

\subsection{Lebesgue measure in the nonstandard universe} 
\label{subsection:Lebesgue} 
We refer to \cite{Kestelman60,Solovay70,CFG91} for some background on Lebesgue measure. 
In this section, we define a nonstandard version of Lebesgue measure. 

Let $A\subseteq [0,1]$ be a nonempty set. Let $B$  be an internal subset of $\Omega$ such that $B\subseteq \st^{-1}(A)$. Note that there may be many internal $B$ contained in $\st^{-1}(A)$.  
Each $B$ has a discrete measure in $\Omega$, which is the discrete probability measure, 
where each $x\in\Omega$ is equally likely. We define the discrete measure of $B$ as 
\[ 
d(B) = \frac{\card(B)}{N+1}, 
\] 
where $N\in {}^*\mathbb{N}\setminus \mathbb{N}$. Note that $d(B)\in {}^*\mathbb{Q} \cap {}^*[0,1]$.

\begin{definition}
\label{defn:lower-upper-Lebesgue-measure} 
Let $A$ be a subset of $[0,1]$, and let $\Omega$ be the set $\{ 0,1/N,\ldots, i/N, \ldots, 1\}$, where $N\in {}^*\mathbb{N}\setminus \mathbb{N}$. 
Let $d(B)$ be defined as above for some internal set $B$. Lower Lebesgue measure of $A$ is defined to be 
\[ 
\underline{\lambda}(A) = \sup\{ \st \: d(B) :\text{internal } B \subseteq \st^{-1}(A) \}, 
\] 
and upper Lebesgue measure of $A$ is defined as 
\[ 
\overline{\lambda}(A) = \inf \{ \st \: d(B) : \text{internal } B \supseteq \st^{-1}(A) \}. 
\] 
 We say $A$ is Lebesgue measurable if $\overline{\lambda}(A) = \underline{\lambda}(A)$, and we write $\lambda(A)$ when $A$ is Lebesgue measurable. 
\end{definition}


\subsection{Hausdorff measure in the nonstandard universe} 
\label{subsection:Hausdorff} 

Hausdorff outer measure is related to Lebesgue outer measure by 
\begin{equation}
\label{eqn:Hausdorff-outer-measure-Lebesgue-outer-measure}
\mathcal{H}^s(A) = c_s \overline{\lambda}(A)  
\end{equation}
for some constant $c_s$ that depends on $s$, and for nice sets $A$. That is, given a nice set embedded in $\mathbb{R}^n$ with positive integral dimension $s$, we use Lebesgue outer measure to find the measure of $A$. 

We now give a discrete version of Hausdorff measure.  
\begin{definition}
\label{defn:h-delta-s}
Given $\delta>0$, $s$ in the unit interval, and $N> \mathbb{N}$, a $\delta$-interval of $\Omega$ is a set $\{i/N,(i+1)/N,\ldots, j/N\}$ with diameter $(j-i+1)/N\leq \delta$. 
For an internal set $B\subseteq \Omega$, the discrete $s$-dimensional measure is defined as 
\begin{equation}
\label{eqn:h-delta-s}
h_{\delta}^s(B) = \min\sum_{i=1}^L (\diam(V_i))^s, 
\end{equation}
where we take the minimum over all partitions $\{ V_1,\ldots, V_L\}$ of $B$ into $\delta$-intervals. 
\end{definition}

The finite $L\in {}^*\mathbb{N}$ varies, depending on the partition. In the discrete space $\Omega$, this definition is internal and since there are finitely-many nonstandard partitions to consider, the min in \eqref{eqn:h-delta-s} is a true minimum. Moreover, since the $V_i$'s are subsets of $\Omega$ which have been normalized, the need for renormalization is unnecessary.

 
\section{Proof of Theorem~\ref{thm:main}}
\label{section:proof-main-theorem}
\begin{proof}
We will prove \ref{item:a} first. Assume that the infinitesimal $\delta$ is fixed, and that $h_{\delta}^s(B)$ is finite for some internal $B\supseteq \st^{-1}(A)$. Let $\eta >0$ be standard. 
Our strategy is to find an $\eta$-cover $U_i$ for $i\in \mathbb{N}$ such that $\sum_i (\diam (U_i))^s \leq \st(h_{\delta}^s(B))$. 

Take an optimal partition $\{ V_j \}_{j=1}^K$ of $B$ in the sense of $h_{\delta}^s(B)$. Using an internal induction in the nonstandard universe, modify $\{ V_j\}_{j=1}^K$ to some other partition $\{ W_k\}_{k=1}^{L}$ as follows. When defining some $W_k$, given an internal interval $I\subseteq \Omega$, consider the leftmost $V_i$ that is to the right of $I$. 
If $I\cup V_i$ is also an interval, replace $I$ with this interval and continue. 
This process stops when one of the following occurs: 
\begin{enumerate}
\item\label{item:process1} $I\cup V_i$ is not an interval, or there is no further $V_i$ to the right of $I$, 
\item\label{item:process2} $\diam(I) \geq \eta$.  
\end{enumerate}

If either \eqref{item:process1} or \eqref{item:process2} happens, we stop the construction and let $W_k = I$. 
We then construct $W_{k+1}$ starting with $I'$, the leftmost $V_i$ to the right of $I$ if any. 

At the end of this construction, we will have nonstandard finitely-many intervals $W_k$ of length at most $\eta + \delta$ partitioning $B$. This gives a countable $\eta$-cover $\mathcal{U}$ of our original $A$ consisting of all sets $U=\st(W_k)$ for some $W_k$ such that the set $U$ has nonempty interior. 

To see that this does indeed cover $A$, consider $a\in A$. Then the monad of $a$ is contained in $B$. By the overspill principle, this monad is contained in an interval $J\subseteq B$, where the length of $J$ is not infinitesimal. 
But then, by construction, the monad of $a$ is contained in either some set $W_k$ or else, in two neighboring sets $W_k$ and $W_{k+1}$ (the latter case is when the construction of the set $W_k$ was ``finished'' whilst inside the monad of $a$). Moreover, $W_k$ (or in the other case both of $W_k$ and $W_{k+1}$) have lengths that are non-infinitesimal, by construction and by choice of $\eta$. So $a$ is either in the interior of $\st(W_k)$ or else is an endpoint of both $\st(W_k)$ and $\st(W_{k+1})$, as required. The cover $\mathcal{U}$ is countable because any set of intervals in the real number line, all with nonempty interior, must necessarily be countable. Finally, since without loss of generality, $s\leq 1$ (since we are working on the unit interval) and hence $(a+b)^s \leq a^s + b^s$ for $a,b > 0$, we have 
\[ 
\sum_{U\in \mathcal{U}} (\diam U)^s = \st\sum_k (\diam (W_k))^s \leq \st \sum_i (\diam (V_i))^s, 
\] 
as required.

We will now prove \ref{item:b}. 
We observe first that for $0< \delta < \eta$, even if $\delta$ is not infinitesimal, then the argument just given shows that $\mathcal{H}_{\eta}^s(A)\leq \st \: h_{\delta}^s(B)$ for all internal $B\supseteq \st^{-1}(A)$. 
Thus we only have to prove the other direction. It suffices to show that for each standard $\delta >0$ and each standard $\varepsilon >0$, there is an internal $B\supseteq \st^{-1}(A)$ with $\mathcal{H}_{\delta}^s(A) \geq \st(h_{\eta}^s(B))-\varepsilon$ for a certain $\eta$ depending only on $\delta$ and $\varepsilon$. 
Let $\lambda = \mathcal{H}_{\delta}^s(A)$. 

Let $\mathcal{U} = \{ U_i\}_{i\in \mathbb{N}}$ be a $\delta$-cover of $A$ such that 
\[ 
\sum_{i=1}^{N}(\diam (U_i))^s \leq \lambda + (\varepsilon/2), 
\]  
and assume without loss of generality that each $U_i$ is an interval. We ``enlarge'' $U_i$ by increasing its length by $(\varepsilon/2)^{1/s} 2^{-i/s}$ 
on each side, obtaining intervals $V_i$. 
By saturation, there is a sequence of intervals $W_i\subseteq \Omega$ such that $\st(W_i)=V_i$ and $W_i$ is defined for all $i< K$, where $K>\mathbb{N}$. Then for any $N>\mathbb{N}$, we have $\bigcup_{i=1}^N W_i \supseteq \st^{-1}(A)$ because of the ``enlarging''. Moreover, for each $N\in \mathbb{N}$, we have 
\begin{align*} 
\sum_{i=1}^{N} (\diam(W_i))^s &\leq \sum_{i=1}^{N} (\diam(U_i) + (\varepsilon/2)^{1/s} 2^{-i/s})^s \\  
&\leq  \sum_{i=1}^N  (\diam(U_i))^s + (\varepsilon/2) + (\varepsilon/2)\sum_{i=1}^N 2^{-i} \\ 
&\leq \lambda + \varepsilon. 
\end{align*} 
By the overspill principle, there is an infinite $N$ such that the above inequalities hold. Thus, some $B=\bigcup_{i=1}^N W_i$ has a partition showing $h_{\eta}^s(B)\leq \lambda + \varepsilon$. 
Finally, note that the maximum diameter of any $W_j$ is $\delta+ (\varepsilon/2)^{1/s} 2^{-1/s}$, which may be made as close to $\delta$ as we like by choosing $\varepsilon$ sufficiently small. This completes the proof. 
\end{proof}

\section{Example to the main theorem}
\label{section:thm-inequality-example}

Unfortunately, it does not seem possible to replace the limit as $\delta\rightarrow 0$ over standard $\delta$ with an infinitesimal $\delta$ in Theorem~\ref{thm:main}\ref{item:b}. In other words, the inequality in Theorem~\ref{thm:main}\ref{item:a} cannot be replaced by an equality, even for carefully chosen $\delta$, as shown by the following example. 

\begin{example}\label{ex:counter-ex-thm-other-inequality}
Consider the Cantor set $C\subseteq [0,1]$ in Example~\ref{example:Cantor-classical-Hausdorff}, which has dimension $s=\log 2/\log 3$.  
Given an internal $B\supseteq \st^{-1}(C)$ and a positive infinitesimal $\delta$, 
we will estimate $h_{\delta}^s(B)$. 

For each $a\in C$, the monad $\st^{-1}(a)$ is covered by an interval of non-infinitesimal length $I\subseteq B$. Taking all such intervals and mapping them back to the unit interval via the standard part map, and taking the interiors of these sets, we obtain an open cover of $C$. Since $C$ is closed and bounded (hence compact), there is a finite subcover. This shows that we may assume that $B$ is one of the sets of $2^m$ intervals of length $3^{-m}$ obtained at the $m$-th stage of the construction of $C$ (if not, such collection would be smaller than $B$). 
For an interval $I\subseteq \Omega$ of length $\ell$, we have that $h_{\delta}^s(I)$ is approximately $(\ell/\delta)\delta^s$, so 
\[ 
h_{\delta}^s(B) = 2^m \frac{3^{-m}}{\delta}\delta^s = (2/3)^m \delta^{s-1}. 
\] 
Since $0<s<1$ and $\delta$ is an infinitesimal, this value is infinite for all standard $m\in \mathbb{N}$. 
Hence $\st (h_{\delta}^s(B))=\infty$ for all internal $B\supseteq \st^{-1}(C)$. 
\end{example} 
 
One can speculate that the problem is that for an infinitesimal $\delta$, the function $h_{\delta}^s$ measures the size of $B$ in terms of $s<1$ , whereas the local structure of $B$ shows that it has dimension $1$. We leave it as future work to determine for which sets we have an equality in Theorem~\ref{thm:main}\ref{item:a} for a suitably chosen infinitesimal $\delta$.

\bibliography{litlist-nonstandard} \label{references}

\begin{thebibliography}{10}

\bibitem{AHFL86}
Sergio Albeverio, Raphael H{\o}egh-Krohn, Jens~Erik Fenstad, and Tom
  Lindstr{\o}m.
\newblock {\em Nonstandard methods in stochastic analysis and mathematical
  physics}, volume 122 of {\em Pure and Applied Mathematics}.
\newblock Academic Press, Inc., Orlando, FL, 1986.

\bibitem{CFG91}
Hallard~T. Croft, Kenneth~J. Falconer, and Richard~K. Guy.
\newblock {\em Unsolved problems in geometry}.
\newblock Problem Books in Mathematics. Springer-Verlag, New York, 1991.
\newblock Unsolved Problems in Intuitive Mathematics, II.

\bibitem{Cutland97}
Nigel~J. Cutland.
\newblock Nonstandard real analysis.
\newblock In {\em Nonstandard analysis ({E}dinburgh, 1996)}, volume 493 of {\em
  NATO Adv. Sci. Inst. Ser. C Math. Phys. Sci.}, pages 51--76. Kluwer Acad.
  Publ., Dordrecht, 1997.

\bibitem{MR2257838}
J\"{u}rgen Elstrodt.
\newblock {\em Ma\ss - und {I}ntegrationstheorie}.
\newblock Springer-Lehrbuch. [Springer Textbook]. Springer-Verlag, Berlin,
  fourth edition, 2005.

\bibitem{Federer69}
Herbert Federer.
\newblock {\em Geometric measure theory}.
\newblock Die Grundlehren der mathematischen Wissenschaften, Band 153.
  Springer-Verlag New York Inc., New York, 1969.

\bibitem{Fremlin-V2}
David~H. Fremlin.
\newblock {\em Measure theory. {V}ol. 2}.
\newblock Torres Fremlin, Colchester, 2003.
\newblock Broad foundations, Corrected second printing of the 2001 original.

\bibitem{Fremlin-V1}
David~H. Fremlin.
\newblock {\em Measure theory. {V}ol. 1}.
\newblock Torres Fremlin, Colchester, 2004.
\newblock The irreducible minimum, Corrected third printing of the 2000
  original.

\bibitem{Fremlin-V3}
David~H. Fremlin.
\newblock {\em Measure theory. {V}ol. 3}.
\newblock Torres Fremlin, Colchester, 2004.
\newblock Measure algebras, Corrected second printing of the 2002 original.

\bibitem{Fremlin-V4}
David~H. Fremlin.
\newblock {\em Measure theory. {V}ol. 4}.
\newblock Torres Fremlin, Colchester, 2006.
\newblock Topological measure spaces. Part I, II, Corrected second printing of
  the 2003 original.

\bibitem{Fremlin-V5}
David~H. Fremlin.
\newblock {\em Measure theory. {V}ol. 5. {S}et-theoretic measure theory. {P}art
  {II}}.
\newblock Torres Fremlin, Colchester, 2015.
\newblock Corrected reprint of the 2008 original.

\bibitem{gorski2012accuracy}
Andrzej G{\'o}rski, Stanis{\l}aw Dro{\.z}d{\.z}, Agnieszka Mokrzycka, and Jakub
  Pawlik.
\newblock {Accuracy Analysis of the Box-Counting Algorithm}.
\newblock {\em Acta Physica Polonica A}, 121(2B), 2012.

\bibitem{Hausdorff18}
Felix Hausdorff.
\newblock Dimension und \"{a}u\ss eres {M}a\ss.
\newblock {\em Math. Ann.}, 79(1-2):157--179, 1918.

\bibitem{Kestelman60}
Hyman Kestelman.
\newblock {\em Modern theories of integration}.
\newblock 2nd revised ed. Dover Publications, Inc., New York, 1960.

\bibitem{mainieri1993equality}
Ronnie Mainieri.
\newblock On the equality of {H}ausdorff and box counting dimensions.
\newblock {\em Chaos: An Interdisciplinary Journal of Nonlinear Science},
  3(2):119--125, 1993.

\bibitem{Ott02}
Edward Ott.
\newblock {\em Chaos in dynamical systems}.
\newblock Cambridge University Press, Cambridge, second edition, 2002.

\bibitem{Robert03}
Alain~M. Robert.
\newblock {\em Nonstandard analysis}.
\newblock Dover Publications, Inc., Mineola, NY, 2003.

\bibitem{Robinson96}
Abraham Robinson.
\newblock {\em Non-standard analysis}.
\newblock Princeton Landmarks in Mathematics. Princeton University Press,
  Princeton, NJ, 1996.
\newblock Reprint of the second (1974) edition, With a foreword by Wilhelmus A.
  J. Luxemburg.

\bibitem{Rogers70}
Claude~A. Rogers.
\newblock {\em Hausdorff measures}.
\newblock Cambridge University Press, London-New York, 1970.

\bibitem{Solovay70}
Robert~M. Solovay.
\newblock A model of set-theory in which every set of reals is {L}ebesgue
  measurable.
\newblock {\em Ann. of Math. (2)}, 92:1--56, 1970.

\end{thebibliography}
\bibliographystyle{plain}

\end{document}